\documentclass{AIP_CP/aip-cp}

\usepackage[numbers]{natbib}
\usepackage{rotating}
\usepackage[utf8]{inputenc}

\usepackage{amsmath}
\usepackage[subnum]{cases}
\usepackage{amssymb}

\usepackage{amsthm}
\usepackage{enumitem}
\usepackage{graphicx,color,subfigure,lscape}
\usepackage[a4paper]{geometry}
\usepackage{etoolbox}
\usepackage{multirow}
\usepackage{booktabs}
\usepackage{xcolor}
\usepackage{bm}

\newtheorem{theorem}{Theorem}[section]

\newcommand{\bs}[1]{\boldsymbol{#1}}
\newcommand{\wh}[1]{\widehat{#1}}
\newcommand{\conv}[1]{\operatorname{conv}({#1})}

\begin{document}

% Title portion
\title{Convex Hull Formulations for \\ Mixed-Integer Multilinear Functions}

\author[aff1]{Harsha Nagarajan}
\author[aff2]{Kaarthik Sundar}
\author[aff1]{Hassan Hijazi}
\author[aff1]{Russell Bent}

\affil[aff1]{Theoretical Division (T-5), Los Alamos National Laboratory, NM, USA (Contact: harsha@lanl.gov)}
\affil[aff2]{Center for Nonlinear Studies, Los Alamos National Laboratory, NM, USA}

\maketitle

\begin{abstract}
In this paper, we present convex hull formulations for a mixed-integer, multilinear term/function (MIMF) that features products of multiple continuous and binary variables. We develop two equivalent convex relaxations of an MIMF and study their polyhedral properties in their corresponding higher-dimensional spaces. We numerically observe that the proposed formulations consistently perform better than state-of-the-art relaxation approaches. 
\end{abstract}
\section{Introduction} \label{sec:intro}
Last few decades has seen tremendous advances in developing efficient convex relaxations of multilinear functions to solve factorable, non-convex Mixed-Integer Nonlinear Programs (MINLPs) to global optimality \cite{bao2015global}. State-of-the-art spatial B\&B-based solvers, such as Baron and SCIP, heavily rely on the tightness of relaxations to efficiently converge to global optimality. Nevertheless, global optimization for large-scale MINLPs is still a challenge.  Particularly, there has been scarce literature which focuses on developing relaxations for Mixed-Integer Multilinear Functions (MIMFs). Thus, the focus of this article is on developing tight, term-wise polyhedral relaxations for an MIMF. 

MIMF is a multilinear function that features products of multiple continuous and binary variables. MINLPs with MIMFs appear in several important applications \cite{lu2017gmd,deep2017acc,bestuzheva2016convex,nagarajan2016optimal}. A hierarchy of reformulation-based relaxations can be applied to generate an extended linear program, however without guarantees of capturing the convex hull in the mixed-integer space \cite{sherali1990hierarchy,liberti2004reformulation}. There has been a line of work that utilizes perspective functions and disjunctive formulations to develop tight and tractable relaxations for convex MINLPs, involving on-off variables \cite{ceria1999convex,gunluk2010perspective,hijazi2012mixed}. In this article, we leverage theses ideas and develop tractable polyhedral relaxations for nonconvex MIMFs. 
%In particular, the convex hull relaxations developed in this article can be derived as non-trivial projections of the perspective relaxations in \cite{ceria1999convex}. 
We now formally define MIMFs.

\textbf{Notation} Boldface fonts denote vectors. Given two vectors $\bm a$ and $\bm b$, $\bm a^T \bm b$ denotes the dot product of the two vectors. $\bm 0$ and $\bm 1$ are used to denote vectors of zeros and ones, respectively. The operation $|\cdot|$ denotes the cardinality of a set. $\conv{A}$ denotes the convex hull of set $A$ containing points in the Euclidean space. \\

\noindent
Given this notation, an MIMF, $\varphi(\bm x, \bm z) : [\bm \ell, \bm u] \times \{0,1\}^{|\mathcal{J}|} \rightarrow \mathbb{R}$, is formally defined as 
$$\varphi(\bm x, \bm z) = \prod_{i\in \mathcal I} x_i \prod_{j \in \mathcal J} z_j,$$ 
where $\mathcal I$ and $\mathcal J$ denote the index set for the vector of continuous and binary variables, respectively, and $[\bm \ell, \bm u] = \{ \bm x \in \mathbb R^{|\mathcal I|}: \bm \ell \leqslant \bm x \leqslant \bm u\}$. For ease of exposition, we also let $\varphi_{\bm x}(\bm x) = \prod_{i \in \mathcal I} x_i$ and $\varphi_{\bm z} (\bm z) = \prod_{j \in \mathcal J} z_j$. We then let $\mathcal K$, indexed by $k$, denote the set of $2^{|\mathcal I|}$ extreme points of $\bm x$ as defined by $[\bm \ell, \bm u]$. The notation $\bm \xi_k$ is used to denote extreme point $k \in \mathcal K$. Given $\varphi(\bm x, \bm z) = \varphi_{\bm x}(\bm x) \cdot \varphi_{\bm z}(\bm z)$, the primary goal of this paper is to develop tight polyhedral relaxations of the graph of such a function, given by the set $X = \{(\bm x, \bm z, \wh{\varphi}) \in [\bm \ell, \bm u] \times \{0,1\}^{|\mathcal{J}|} \times \mathbb R: \wh{\varphi} = \varphi(\bm x, \bm z) \}$. 
% \label{eq:graph}

% \section{MIMF special cases} 
% \label{sec:lit}

% \subsection{Multilinear functions of continuous variables} \label{subsec:ml-c}
\paragraph{McCormick Relaxations:} In the special case of continuous bilinear function, $\varphi_{\bm x}(\bm x) = x_1 x_2$, McCormick \cite{mccormick1976computability} developed a widely used tight relaxation. Given set
$X_{\bm x}^B = \left\{ (x_1, x_2, \wh{\varphi}) \in [\ell_1, u_1] \times  [\ell_2, u_2] \times \mathbb R: \wh{\varphi} = x_1 x_2 \right\}$, the McCormick relaxation is given by 
\begin{equation}
\left \langle x_1,x_2,\wh \varphi \right\rangle^{MC} := \conv{X^B_{\bm x}} = 
 \left\{
\begin{gathered}
    (x_1,x_2,\wh{\varphi})\in  [\ell_1,u_1]\times [\ell_2,u_2] \times \mathbb{R}: \\
    \wh{\varphi} \geqslant u_2 x_1 + u_1 x_2 - u_1 u_2, \quad \wh{\varphi} \geqslant \ell_2 x_1 + \ell_1 x_2 - \ell_1 \ell_2, \\
   \wh{\varphi} \leqslant u_2 x_1 + \ell_1 x_2 - \ell_1 u_2, \quad \wh{\varphi} \leqslant \ell_2 x_1 + u_1 x_2 - u_1 \ell_2
\end{gathered}
\right\}.
\label{eq:mc}
\end{equation}
It is known that Eq. \eqref{eq:mc} describes the convex hull of $X_{\bm x}^B$ (see \cite{mccormick1976computability}). For a multilinear function, $\varphi_{\bm x}(\bm x) = \prod_{i\in\mathcal{I}}x_i$, standard global optimization methods apply recursive McCormick relaxations sequentially on bilinear terms, which do not necessarily capture the convex hull of the graph of $\varphi_{\bm x}(\bm x)$ \cite{nagarajan2016tightening,nagarajan2017adaptive}. 

In the special case of multilinear function, $\varphi_{\bm z}(\bm z) = \prod_{i\in\mathcal{J}}z_i$, \cite{fortet1960applications} derives a set of $|\mathcal{J}|+1$ number of constraints which captures the convex hull of the graph of $\varphi_{\bm z}(\bm z)$. Given the set $X_{\bm z} = \left\{(\bm z,\wh{z}) \in \{0,1\}^{|\mathcal{J}|} \times [0,1]: \wh{z} = \varphi_{\bm z}(\bm z)\right\}$, it's exact linear reformulation is given by 
 \begin{equation}
 \left\{ (\bm z,\widehat{z}) \in \{0,1\}^{|\mathcal{J}|} \times [0,1]: \wh{z}\cdot \bm 1 \leqslant \bm z \  \text{ and } \wh{z} \geqslant \bm 1^T \bm z - |\mathcal{J}| + 1\right \}.
 \label{eq:mcc_bin}
 \end{equation}

% \paragraph{Recursive McCormick Relaxations}: For MIMF 
% that consists of continuous variables of the form $\varphi_{\bm x}(\bm x) = \prod_{i\in\mathcal{I}}x_i$, and is defined by the set
% \begin{flalign}
% X_{\bm x} = \left\{ (\bm x, \wh{\varphi}) \in [\bm \ell, \bm u]  \times \mathbb R: \wh{\varphi} = \varphi_{\bm x}(\bm x) \right\}, \label{eq:bilinear}
% \end{flalign}
% McCormick proposed a relaxation that is constructed by recursively, {after the introduction of sufficient number of auxiliary variables,} using Eq. \eqref{eq:mc} as
% \begin{equation}
% \left \langle \varphi_{\bm x}(\bm x),\bm{\wh\varphi} \right \rangle^{RMC} = \left\langle\left\langle\left\langle x_1,x_2,\wh\varphi_1\right\rangle^{MC},\ldots, x_{|\mathcal{I}|-1},\wh\varphi_{|\mathcal{I}|-2}\right\rangle^{MC},x_{|\mathcal{I}|},\wh\varphi_{|\mathcal{I}|-1} \right\rangle^{MC}.
% \label{eq:mcc_ml}
% \end{equation}
% We refer to this relaxation as the ``Recursive McCormick'' relaxation. The recursive McCormick relaxation has formed the basis for handling multilinear terms in many global optimization algorithms. However, this relaxation is not guaranteed to form the convex hull of $X_{\bm x}$ \cite{luedtke2012some}. As a result, recent work has focused on tightening recursive McCormick relaxations by applying associativity in different ways. An interested reader is referred to  \cite{belotti2013composition,nagarajan2017adaptive,nagarajan2016tightening} for recent progress on composition of convex envelopes of multilinear functions. 
\paragraph{$\bm \lambda$-Formulation:} The strongest relaxation of the set $X_{\bm x} = \{(\bm x, \wh{\varphi}) \in [\bm \ell, \bm u] \times \mathbb R: \wh{\varphi} = \varphi(\bm x) \}$ is based on the extreme-point characterization.
More formally, \cite{Rikun1997} shows that $\conv{X_{\bm x}}$ is defined by 
\begin{gather}
\conv{X_{\bm x}} = \underset{\bm x,\wh{\varphi}}{\mathrm{Proj}} \left( \left \langle \varphi_{\bm x}(\bm x),\wh{\varphi} \right\rangle^{\bm \lambda} \right), \label{eq:conv_lambda} \\ 
\left \langle \varphi_{\bm x}(\bm x),\wh{\varphi} \right \rangle^{\bm \lambda}  = 
\left\{ (\bm x, \wh{\varphi}, \bm \lambda)  \in [\bm \ell, \bm u] \times \mathbb R \times \Delta_{|\mathcal K|} : \bm x = \sum_{k \in \mathcal K} \lambda_k \bm \xi_k, \,\wh{\varphi} = \sum_{k\in \mathcal K} \lambda_k \varphi_{\bm x}(\bm \xi_k) \right\} \label{eq:lambda_ml}
\end{gather}

$\Delta_{|\mathcal K|}$ is a $|\mathcal K|$-dimensional 0-1 simplex. This formulation can also be applied to an MIMF ($\varphi(\bm x,\bm z)$) with $2^{|\mathcal{I}|+|\mathcal{J}|}$ extreme points. However, the major drawback of using this formulation for MIMF is that $|\bm \lambda|$ grows exponentially with the number terms in the MIMF. Thus, it would be useful to derive more compact relaxations. In summary, the literature on tractable formulations for an MIMF, $\varphi(\bm x, \bm z)$, is scarce, thus motivating the study in this article.

\section{Disjunctive formulation $\mathfrak{F}^{\lambda}$}
% \subsection{Formulation $\mathfrak{F}^{\lambda}$}
We now present $\mathfrak{F}^{\lambda}$, a relaxation of $X$ based on an extreme-point characterization that has at most $\mathcal{O}(2^{|\mathcal{I}|})$ $\bm \lambda$ variables. This relaxation  is obtained as a disjunctive union of sets $X^0$ and $X^1$, where
%----- X0 eqn -----%
%\begin{gather*}
%X^0 = \left\{
%\begin{gathered}
%(\bm x,\wh{\varphi}, \bm \lambda, \bm z, \wh{z}) \in [\bm \ell, \bm u] \times \mathbb R \times [0,1]^{|\mathcal K| + |\mathcal J| + 1} : \\ 
%\wh{\varphi} = 0, \ \bm 1^T \bm z \leqslant |\mathcal{J}| - 1, \ \wh{z} = 0, \ \bm \lambda = \bm 0
%\end{gathered}
%\right\}, \ 
%X^1 =  \left\{
%\begin{aligned}
%(\bm x,\wh{\varphi}, \bm \lambda, \bm z, \wh{z}) \in \langle \varphi_{\bm x}(\bm x), \wh{\varphi}\rangle^{\bm \lambda} \times [0,1]^{|\mathcal J| + 1} :  \bm z = \bm 1, \ \wh{z} = 1 
%\end{aligned}
%\right\}. 
%\label{eq:x0} 
%\end{gather*}

% %----- X1 eqn -----%
% \begin{equation}
% X^1 =  \left\{
% \begin{aligned}
% (\bm x,\wh{\varphi}, \bm \lambda, \bm z, \wh{z}) \in \langle \varphi_{\bm x}(\bm x), \wh{\varphi}\rangle^{\bm \lambda} \times [0,1]^{|\mathcal J| + 1} :  \bm z = \bm 1, \ \wh{z} = 1 
% \end{aligned}
% \right\}. 
% \label{eq:x1}
% \end{equation}
% \noindent  We now define the set $\wh X$ as

%----- X0 eqn -----%
\begin{gather}
X^0 = \left\{
\begin{gathered}
(\bm x,\wh{\varphi}, \bm \lambda, \bm z, \wh{z}) \in [\bm \ell, \bm u] \times \mathbb R \times [0,1]^{|\mathcal K| + |\mathcal J| + 1} : \\ 
\wh{\varphi} = 0, \ \bm 1^T \bm z \leqslant |\mathcal{J}| - 1, \ \wh{z} = 0, \ \bm \lambda = \bm 0
\end{gathered}
\right\} 
\label{eq:x0} 
\end{gather}

%----- X1 eqn -----%
\begin{equation}
X^1 =  \left\{
\begin{aligned}
(\bm x,\wh{\varphi}, \bm \lambda, \bm z, \wh{z}) \in \langle \varphi_{\bm x}(\bm x), \wh{\varphi}\rangle^{\bm \lambda} \times [0,1]^{|\mathcal J| + 1} :  \bm z = \bm 1, \ \wh{z} = 1 
\end{aligned}
\right\}. 
\label{eq:x1}
\end{equation}
\noindent  We now define the set $\wh X$ and show that 
$\wh X$ is a formulation for $\operatorname{conv}(X^0 \cup X^1)$.

%----- Conv(X0 U X1) ------%
\begin{equation}
\wh X = \left\{ 
\begin{gathered}
(\bm x,\wh{\varphi}, \bm \lambda, \bm z, \wh{z}) \in [\bm \ell, \bm u] \times \mathbb R \times [0,1]^{|\mathcal K| + |\mathcal J| + 1} : \ \\
\wh{z}\cdot\bm 1 \leqslant \bm z, \ \wh{z} \geqslant \bm 1^T \bm z - |\mathcal{J}|+1, \\
\bm 1^T \bm \lambda = \wh{z}, \, \wh{\varphi} = \sum_{k \in \mathcal{K}} \lambda_k \varphi_{\bm x}(\bm \xi_k), \\ 
\sum_{k \in \mathcal{K}} \lambda_k \bs{\xi}_k + \bm \ell (1-\wh{z}) \leqslant \bm x \leqslant \sum_{k \in \mathcal{K}} \lambda_k \bm \xi_k +  \bm u(1-\wh{z}) \\
\end{gathered}
\right\}
\label{eq:x0Ux1}
\end{equation}

\begin{theorem} 
\label{thm:x}
$\operatorname{conv}(X^0 \cup X^1) = \wh X$.
\end{theorem}
\begin{proof}
This theorem is proved by showing $\operatorname{conv}(X^0 \cup X^1) \subseteq \wh X$ and $ \wh X \subseteq \operatorname{conv}(X^0 \cup X^1)$ are true. 
   \paragraph{$\boxed{\operatorname{conv}(X^0 \cup X^1) \subseteq \wh X}$} 
   We first observe that when $\wh z = 0$, $X^0 = \wh{X}$, indicating that $X^0 \subset \wh X$. Similarly, we have $X^1 \subset \wh X$. These facts together with $\wh X$ being a convex set results in $\operatorname{conv}(X^0 \cup X^1) \subseteq \wh X$. 
\paragraph{$\boxed{ \wh X \subseteq \operatorname{conv}(X^0 \cup X^1)}$}
Let $\bm p^* = (\bm x^*, \wh{\varphi}^*, \bm \lambda^*, \bm z^*, \wh{z}^*)$ be any point in $\wh{X}$. If $\wh{z}^* = 0$, then $\bm p^* = (\bm x^*, 0, \bm 0, \bm z^*, 0)$ and $\bm p^* \in X^0$. Similarly, if $\wh{z}^* = 1$, it is trivial to conclude that $\bm p^* \in X^1$. 
Now consider the case where $0 < \wh{z}^* < 1$. To show $\bm p^* \in \operatorname{conv}(X^0 \cup X^1)$, we construct two points $\bm p^*_0 \in X^0$ and $\bm p^*_1 \in X^1$ such that $\bm p^*$ is a convex combination of $\bm p^*_0$ and $\bm p^*_1$. Given $\bm 1 \in \mathbb R^{|\mathcal J|}$, the points are
%, which would in turn prove that $\wh X \subseteq \operatorname{conv}(X^0 \cup X^1)$. 
    \begin{flalign}
    \bm p^*_0 = \left(\frac{\bm x^* - \sum_{k \in \mathcal{K}} \lambda_k^* \bm{\xi}_k}{1-\wh{z}^*}, 0, \bm 0, \frac{\bm z^* - \wh{z}^* \cdot \bm 1}{1-\wh{z}^*} , 0\right), \label{eq:p0} \quad 
    \bm p^*_1 = \left( \sum_{k \in \mathcal{K}} \left(\frac{\lambda_k^*}{\wh{z}^*}\right) \bs{\xi}_k, \frac{\wh{\varphi}^*}{\wh{z}^*}, \frac{\bm \lambda^*}{\wh{z}^*}, \bm 1, 1\right).
    \end{flalign} 
    \noindent
    To show $\bm p^*_0 \in X^0$, we only need to prove that
    \begin{flalign}
    \bm \ell \leqslant \frac{\bm x^* - \sum_{k \in \mathcal{K}} \lambda_k^* \bs{\xi}_k}{1-\wh{z}^*} \leqslant \bm u \quad \text{and} \quad \frac{ \bm 1^T \bm z^* - |\mathcal J| \wh{z}^*}{1-\wh{z}^*} \leqslant |\mathcal J| - 1. \label{eq:p0_claim}
    \end{flalign}
    is true.
    %Eq. \eqref{eq:p0_claim} is trivial to observe from the following inequalities that stem from the assumption
    Given that $\bm p^* \in X$ and $0 < \wh{z}^* < 1$, then the following inequalities hold:
    \begin{gather}
    \sum_{k \in \mathcal{K}} \lambda^*_k \bs{\xi}_k + \bm \ell (1-\wh{z}^*) \leqslant \bm x^* \leqslant \sum_{k \in \mathcal{K}} \lambda_k^* \bm \xi_k +  \bm u(1-\wh{z}^*)  \label{eq:xvalid} \\
    \wh{z}^* \geqslant \bm 1^T \bm z^* - |\mathcal J| + 1. \label{eq:zvalid}
    \end{gather}
    The inequalities of Eq. \eqref{eq:p0_claim} are then derived through linear algebra on Eqs. \eqref{eq:xvalid} and \eqref{eq:zvalid}. Similarly, it is easy to verify that the above point $\bm p_1^* \in X^1$. 

    \noindent $\wh X \subseteq \operatorname{conv}(X^0 \cup X^1)$ is proved by observing that $\bm p^* = (1-\wh{z}^*)\bm p_0^* + \wh{z}^* \bm p_1^*$ when $0 < \wh{z}^* < 1$. 
    Given  $\operatorname{conv}(X^0 \cup X^1) \subseteq \wh X$ and $ \wh X \subseteq \operatorname{conv}(X^0 \cup X^1)$, we have $\operatorname{conv}(X^0 \cup X^1) = \wh X$.
\end{proof}

\noindent 
Note that $\mathfrak{F}^{\lambda}$ can be obtained using a non-trivial projection of the extended formulation described in \cite{ceria1999convex} characterizing the convex hull of the union of convex sets using perspective maps.
We \emph{conjecture} that the projection of $\mathfrak{F}^{\lambda}$ onto the space of original variables is indeed $\operatorname{conv}(X)$, but this remains an open question. 

\subsection{Disjunctive formulation $\mathfrak{F}^{rmc}$}
We now present $\mathfrak{F}^{rmc}$, a relaxation of set $X$ based on recursive McCormick relaxations on bilinear functions, typically employed in state-of-the-art global solvers. Let
\begin{equation}
  \varphi(\bm x, \bm z) = \underbrace{x_1x_2\dots x_{|\mathcal{I}|-1}}_{\tilde\varphi_{\bm{x}}(\bm x)}x_{|\mathcal{I}|} \prod_{j \in \mathcal{J} }z_j.
  \label{eq:f_rmc}
\end{equation}
By applying recursive McCormick relaxations, $\tilde\varphi_{\bm x}(\bm x)$ can be replaced by a lifted variable $\wh \varphi$, thus reducing \eqref{eq:f_rmc} to $\varphi(\bm x, \bm z) = \wh\varphi \cdot x_{|\mathcal{I}|} \prod_{j \in \mathcal{J} }z_j$. These recursive relaxations do not necessarily capture the convex hull of the MIMF in the extended space. However, observing the special structure of $\wh\varphi \cdot x_{|\mathcal{I}|} \prod_{j \in \mathcal{J} }z_j$, which is bilinear in continuous variables and multilinear in binary variables, we now characterize the convex hull of $\varphi(\bm x,\bm z) = x_1x_2\prod_{j\in \mathcal{J}}z_j,$ in the extended space of McCormick-based constraints. The formulation is obtained as a disjunctive union of the sets $Z^0$ and $Z^1$, given by:

%----- Z0 eqn -----%
\begin{gather}
Z^0 = \left\{
\begin{gathered}
(\bm x,\wh{\varphi}, \bm{\wh{xz}}, \bm z, \wh{z}) \in [\bm \ell, \bm u] \times \mathbb R^3 \times [0,1]^{|\mathcal J| + 1} : \\ 
\wh{\varphi} = 0, \ \bm{\wh{xz}}=\bm 0, \ \bm 1^T \bm z \leqslant |\mathcal{J}| - 1, \ \wh{z} = 0
\end{gathered}
\right\} 
\label{eq:z0} 
\end{gather}

%----- Z1 eqn -----%
\begin{gather}
Z^1 =  \left\{
\begin{gathered}
(\bm x,\wh{\varphi}, \bm{\wh{xz}}, \bm z, \wh{z}) \in \langle \bm x, \wh \varphi \rangle^{MC} \times \mathbb R^2 \times [0,1]^{|\mathcal J| + 1} : \\
\bm z = \bm 1, \ \wh{z} = 1,  \ \bm{\wh{xz}}=\bm x
\end{gathered}
\right\}. 
\label{eq:z1}
\end{gather}

\paragraph{} We now define the set $\wh Z$ as follows:
%----- Conv(Z0 U Z1) ------%
\begin{equation}
\wh Z = \left\{ 
\begin{gathered}
(\bm x,\wh{\varphi}, \bm{\wh{xz}}, \bm z, \wh{z}) \in [\bm \ell, \bm u] \times \mathbb R^{3} \times [0,1]^{|\mathcal J| + 1}:  \\
\wh{z}\cdot\bm 1 \leqslant \bm z, \ \wh{z} \geqslant \bm 1^T \bm z - |\mathcal{J}|+1, \\
\wh{\varphi} \geqslant u_2 \cdot \wh{xz}_1 + u_1 \cdot\wh{xz}_2 - u_1 u_2\cdot\wh z, \quad \wh{\varphi} \geqslant \ell_2 \cdot\wh{xz}_1 + \ell_1 \cdot\wh{xz}_2 - \ell_1 \ell_2\cdot\wh z, \\
\wh{\varphi} \leqslant u_2 \cdot\wh{xz}_1 + \ell_1 \cdot\wh{xz}_2 - \ell_1 u_2\cdot\wh z, \quad \wh{\varphi} \leqslant \ell_2 \cdot\wh{xz}_1 + u_1 \cdot\wh{xz}_2 - u_1 \ell_2\cdot\wh z, \\ 
\wh z  \cdot \bm \ell \leqslant \bm{\wh{xz}} \leqslant \wh z  \cdot \bm u , \\
\bm x - (1-\wh z) \cdot \bm u \leqslant \bm{\wh{xz}} \leqslant \bm x - (1- \wh z) \cdot \bm \ell
\end{gathered}
\right\}
\label{eq:z0Uz1}
\end{equation}

\begin{theorem} 
\label{thm:z}
$\operatorname{conv}(Z^0 \cup Z^1) = \wh Z$.
\end{theorem}
\begin{proof}
The proof of this theorem is very similar to that of Theorem \ref{thm:x}. 
\end{proof}
A special case of this function, i.e., $\varphi=x_1x_2z$ has been dealt with in \cite{bestuzheva2016convex}, where the authors derive a convex hull formulation in the space of original variables only when the variables $(x_1,x_2)$ are forced to zero values when the binary variable $z$ is assigned a zero value. 

\section{Initial Results and Conclusions} 
% \todo{Not addressed minor comments 7, 8, and 9 of reviewer 2. Will do after you finish your additions to this section.}
\label{sec:res}
All formulations were solved on a laptop with an Intel(R) i7, 2.60GHz processor and 16GB of memory using Gurobi 7.5.2 with default options. 
As shown in \eqref{eq:ex1}, we consider an MINLP with a sum of $2k-$linear MIMFs such that the feasible set is non-empty and admits non-trivial solutions. 
\begin{equation}
  \begin{aligned}
    &  \underset{\bm x, \bm z}{\text{minimize}}
    & & \sum_{i=1}^{n} \left(c_ix_i + d_iz_i\right) \\
    &  \text{subject to}
    & & \sum_{i=1}^{n-k+1}\left(\prod_{j=i}^{i+k-1} x_jz_j\right) \geqslant D, \\
    & & & \bm x \in [\bm \ell, \bm u], \quad \bm z \in \{0,1\}^{n}
    \end{aligned}
    \label{eq:ex1}
\end{equation}

In \eqref{eq:ex1}, $c_i,d_i$ and $\bm \ell$ are independently assigned pseudorandom values on an open interval between 0 and 1, and $\bm u = 10\bm \ell$. For example, at $k=4$ and $n=5$, the constraint in \eqref{eq:ex1} will be $x_1x_2x_3x_4z_1z_2z_3z_4 +x_2x_3x_4x_5z_2z_3z_4z_5 \geqslant D$. Applying $\mathfrak{F}^{\lambda}$ on \eqref{eq:ex1} implies that every MIMF is replaced by a lifted variable, say $\wh \varphi_i$, such that $\sum_{i} \wh \varphi_i \geq D$ and $ \wh \varphi_i$ admits the constraints from  $\mathfrak{F}^{\bm \lambda}$ in \eqref{eq:x0Ux1}, thus creating a lower-bounding MILP for \eqref{eq:ex1}. For this MILP, we define  
$$\mathrm{LP} \ \mathrm{gap} = \frac{OPT-LB}{OPT}\cdot 100,$$
where, $OPT$ and $LB$ correspond to MILP's optimal and continuous relaxation objective values.

 \begin{table}[htp]
% \begin{wraptable}{r}{10cm}
%\scriptsize
% \footnotesize
		\centering
				\caption{$\mathfrak{F}^{\lambda}$ and $\mathfrak{F}^{rmc}$ applied on \eqref{eq:ex1} for $k=4$ and $D = 0.7n$. Bold font represents best run times.}
		\label{tab:quad}
		\setlength{\tabcolsep}{6pt}
		\centering
    \begin{tabular}{ccccccccc}
        \hline\noalign{\smallskip}
        $n$& \multicolumn{2}{c}{MILP obj.} & \multicolumn{2}{c}{LP gap (\%)} & \multicolumn{2}{c}{LP run time (sec.)} & \multicolumn{2}{c}{MILP run time (sec.)} \\ 
        \cmidrule(lr){2-3} \cmidrule(lr){4-5} \cmidrule(lr){6-7} \cmidrule(lr){8-9} 
         & $\mathfrak{F}^{\lambda}$ &  $\mathfrak{F}^{rmc}$& $\mathfrak{F}^{\lambda}$&  $\mathfrak{F}^{rmc}$ & $\mathfrak{F}^{\lambda}$  & $\mathfrak{F}^{rmc}$ & $\mathfrak{F}^{\lambda}$ & $\mathfrak{F}^{rmc}$\\
		\noalign{\smallskip}
\hline
\noalign{\smallskip}
100	&	366.0	&	365.58	&	3.1	&	3.1	&	\bf{2.6}&	2.8	&	3.0	&	\bf{2.7}	\\
500	&	1750.4	&	1750.3	&	0.4	&	0.5	&	3.5	&	\bf{2.5}	&	\bf{3.4} &	4.0	\\
1000	&	3484.7	&	3484.7	&	0.1	&	0.6	&	3.0	&	\bf{2.5}	&	\bf{5.8}&	8.4	\\
2000	&	6804.7	&	6802.5	&	$<$0.001	&	0.4	&	\bf{2.8}		&	6.6	&	\bf{4.5}&	11.2	\\
4000	&	13503.4	&	13500.4	&	$<$0.001	&	0.4	&	\bf{3.9}&	7.8	&	\bf{6.1}		&	13.2	\\
6000	&	20151.2	&	20148.2	&	$<$0.001&	0.3	&	\bf{5.3}&	10.2	&	\bf{45.1}		&	201.2	\\
8000	&	27121.9	&	27117.8	&	$<$0.001&	0.3	&	\bf{6.2}	&	12.2	&	\bf{6.5}		&	198.0	\\
10000	&	33905.6	&	33900.5	&	$<$0.001	&	0.4	&	\bf{20.6}	&	23.5	&	\bf{62.0}	&	540.2	\\
        \noalign{\smallskip}
        \hline
		\end{tabular}
%\end{wraptable}
\end{table}
Table \ref{tab:quad} shows the performance comparisons of formulations,  $\mathfrak{F}^{\lambda}$ and $\mathfrak{F}^{rmc}$ on double-quadrilinear functions. Clearly, $\mathfrak{F}^{\lambda}$ overall performs the best in terms of runtimes, both for MILPs and their LP relaxations. Though $\mathfrak{F}^{\lambda}$ does not necessarily capture the convex hull of an arbitrary sum of MIMFs, the LP gaps are very tight and indeed produce close-to integral solutions (gaps $<$ 0.001\%) on large instances, thus speeding up the convergence of MILPs. Another interesting observation is that $\mathfrak{F}^{rmc}$, with recursive McCormick relaxations clearly loses on capturing the convex hull of individual MIMFs, thus producing weaker lower bounds to the original MINLP in \eqref{eq:ex1}. However, for $k=2$ (double-bilinear) in \eqref{eq:ex1}, we observed that both $\mathfrak{F}^{\lambda}$ and $\mathfrak{F}^{rmc}$ produced identical lower bounds, validating the bilinear convex hull result in \eqref{eq:z0Uz1}. Further, for $k=4$, though the number of extreme points grow up to ``16'' per MIMF, the numerical performance of $\mathfrak{F}^{\lambda}$ is still superior to $\mathfrak{F}^{rmc}$, which closely represents state-of-the-art relaxation approaches applied in the literature.  

\paragraph{Summary} In this paper, we considered MIMFs and developed new convex relaxations based on the convex hull of the disjunctive union of these two sets of variables in an extended space. While this paper has made strides in tightening relaxations of MIMFs, there remain a number of important future directions including the characterization of the convex hull of an MIMF in the space of its original variables.

\paragraph{Acknowledgements} The work was funded by the Center for Nonlinear Studies (CNLS) at LANL and the LANL's directed research and development
project “POD: A Polyhedral Outer-approximation, Dynamic-discretization optimization solver”. It was carried out under the auspices of the NNSA of the U.S. DOE at 
LANL under Contract No. DE-AC52-06NA25396.

% \end{table}

% \begin{wrapfigure}{r}{.28\textwidth}
% \begin{equation}
% \small
%   \begin{aligned}
%     &  \underset{\bm x, \bm z}{\text{minimize}}
%     & & \sum_{i=1}^{n} \left(c_ix_i + d_iz_i\right) \\
%     &  \text{subject to}
%     & & \sum_{i=1}^{n-k+1}\left(\prod_{j=i}^{i+k-1} x_jz_j\right) \geqslant D, \\
%     & & & \bm x \in [\bm \ell, \bm u], \ \bm z \in \{0,1\}^{n}
%     \end{aligned}
%     \label{eq:ex1}
% \end{equation}
% \end{wrapfigure}

% \nocite{*}
\bibliographystyle{plain}%
% \small
\bibliography{references.bib}

\end{document}